\documentclass{amsart}

\usepackage{epsfig}
\usepackage{psfrag}

\usepackage{infix-RPN}
\usepackage{pst-plot,pst-infixplot,pstricks,graphicx}
\usepackage{amssymb,latexsym,amsthm,amsfonts,color,fancyhdr}

\newtheorem{theorem}{Theorem}[section]

\newtheorem{lemma}[theorem]{Lemma}

\newtheorem{definition}[theorem]{Definition}

\theoremstyle{remark}
\newtheorem{remark}[theorem]{Remark}

\numberwithin{equation}{section}

\begin{document}

\title[Classical Orthogonal polynomials on $q$-quadratic lattices]
{Remarks on the paper "Recurrence equations and their classical orthogonal polynomial solutions on a quadratic or a q-quadratic lattice"}
\author{D. Mbouna}
\address{Department of Mathematics, Faculty of Sciences, University of Porto, Campo Alegre st., 687, 4169-007 Porto, Portugal}
\email{dieudonne.mbouna@fc.up.pt}



\date{\today}



\date{\today}


\maketitle


\begin{abstract}
We provide a simple method to recognize classical orthogonal polynomials on lattices defined only by their coefficients of the three term recurrence relation.  
\end{abstract}

\section{Introduction}

Despite all investigations done in the literature on classical orthogonal polynomial sequences (OPS) there are still some interesting and challenging problems related to them. For instance, D. Abdulaziz Alhaidari (see \cite{ADA}) submitted as open problem for the proceedings of the OPSFA-14 conference two families of orthogonal polynomials on the real line only described by their three term recurrence relation (TTRR). These OPS were obtained in some previous works by him with co-authors and due to the prime significance of these polynomials in physics, he posed the question the derivation of their other properties like the weight functions, generating functions, orthogonality, possible hypergeometric froms, Rodrigues-type formulas and many others. Yutian Li noticed that one of the polynomial sequence was a special case of the Wilson polynomial (see \cite[p.5]{Walter}) and subsequently this problem was solved in \cite{Walter} by W. Van Assche where he identified some of them (including special cases and some asymptotics) as classical OPS. He also provided therein some ideas (\cite[p.2]{Walter}) to recognize classical OPS based on their coefficients of the TTRR. These ideas are in general non trivial (and a fortiori for classical OPS on lattices) depending on how the polynomial is defined. Although D. Tcheutia in \cite{DDT} provided a general method answering this problem, this inspires us to write this short note where we give another ideas. This is motived by the fact that the answer given by D. Tcheutia consists of checking if the polynomial can be written in some appropriate bases and then use an algorithm/package and so computer system to identify them among classical ones. With our characterization (see Theorem \ref{T1} and Theorem \ref{T2}), this type of problem can be solved in a very simple way and by hands. This also reveals some asymptotic behaviour and relation between coefficients of the TTRR of such OPS. We focus only on classical OPS with respect to the Askey-Wilson operator (so on $q$-quadratic lattices) since, from this, we can obtained similar results for classical OPS on quadratic lattices as discussed in \cite{KMP2022}.

\section{Preliminaries}
Let $\mathcal{P}$ be the vector space of all polynomials with complex coefficients
and let $\mathcal{P}^*$ be its algebraic dual. A simple set in $\mathcal{P}$ is a sequence $(P_n)_{n\geq0}$ such that $\mathrm{deg}(P_n)=n$ for each $n$. A simple set $(P_n)_{n\geq0}$ is called an OPS with respect to ${\bf u}\in\mathcal{P}^*$ if 
$$
\langle{\bf u},P_nP_m\rangle=\kappa_n\delta_{n,m}\quad(m=0,1,\ldots;\;\kappa_n\in\mathbb{C}\setminus\{0\}),
$$
where $\langle{\bf u},f\rangle$ is the action of ${\bf u}$ on $f\in\mathcal{P}$. In this case, we say that ${\bf u}$ is  regular. The left multiplication of a functional ${\bf u}$ by a polynomial $\phi$ is defined by
$$
\left\langle \phi {\bf u}, f  \right\rangle =\left\langle {\bf u},\phi f  \right\rangle \quad (f\in \mathcal{P}).
$$

We consider the Askey-Wilson operator defined by
\begin{align}
(\mathcal{D}_q  f)(x)=\frac{\breve{f}\big(q^{1/2} z\big)
-\breve{f}\big(q^{-1/2} z\big)}{\breve{e}\big(q^{1/2}z\big)-\breve{e}\big(q^{-1/2} z\big)},\quad
z=e^{i\theta}, \label{0.3}
\end{align}
where $\breve{f}(z)=f\big((z+1/z)/2\big)=f(\cos \theta)$ for each polynomial $f$ and $e(x)=x$.
Here $0<q<1$ and $\theta$ is not necessarily a real number (see \cite[p.\,300]{I2005}). Recall that a monic OPS, $(P_n)_{n\geq 0}$, satisfies the following three term recurrence relation (TTRR):
\begin{align}
x P_n(x)=P_{n+1}(x)+B_nP_n(x)+C_nP_{n-1}(x), \quad n=0,1,2,\dots,\label{TTRR}
\end{align}
with $P_{-1}(x)=0$ and $B_n\in \mathbb{C}$ and $C_{n+1} \in \mathbb{C}\setminus \left\lbrace 0\right\rbrace$ for each $n=0,1,2,\ldots$.  Hereafter, we denote $x=x(s)=(q^{s}+q^{-s})/2$ with $0<q<1$. Taking $e^{i\theta}=q^s$ in \eqref{0.3}, $\mathcal{D}_q$ reads
\begin{equation*}
\mathcal{D}_q f(x(s))= \frac{f\big(x(s+\frac{1}{2})\big)-f\big(x(s-\frac{1}{2})\big)}{x(s+\frac{1}{2})-x(s-\frac{1}{2})}.
\end{equation*}
We define an operator $\mathcal{S}_q$ by
\begin{equation*}
\mathcal{S}_q f(x(s))=\frac{f\big(x(s+\frac{1}{2})\big)+f\big(x(s-\frac{1}{2})\big)}{2}.
\end{equation*}
The Askey-Wilson and the averaging operators  induce two elements on $\mathcal{P}^*$, say $\mathbf{D}_q$ and $\mathbf{S}_q$, via the following definition (see \cite{F}): 
\begin{align*}
\langle \mathbf{D}_q{\bf u},f\rangle=-\langle {\bf u},\mathcal{D}_q f\rangle,\quad \langle\mathbf{S}_q{\bf u},f\rangle=\langle {\bf u},\mathcal{S}_q f\rangle.
\end{align*}
Define
\begin{align*}
\alpha= \frac{q^{1/2}+q^{-1/2}}{2},\quad \alpha_n= \frac{q^{n/2}+q^{-n/2}}{2},\quad  \gamma_n=\frac{q^{n/2}-q^{-n/2}}{q^{1/2}-q^{-1/2}}, \quad n=0,1,\dots\;.
\end{align*}

\section{Main results}

The monic Askey-Wilson polynomial, $(Q_n(\cdot; a_1, a_2, a_3, a_4 | q))_{n\geq 0}$, satisfy \eqref{TTRR} (see \cite[(14.1.5)]{KLS2010}) with
\begin{align*}
2B_n &= a_1+\frac{1}{a_1}-\frac{(1-a_1a_2q^n)(1-a_1a_3q^n)(1-a_1a_4q^n)(1-a_1a_2a_3a_4q^{n-1})}{a_1(1-a_1a_2a_3a_4q^{2n-1})(1-a_1a_2a_3a_4q^{2n})}\\[7pt]
&\quad-\frac{a_1(1-q^n)(1-a_2a_3q^{n-1})(1-a_2a_4q^{n-1})(1-a_3a_4q^{n-1})}{(1-a_1a_2a_3a_4q^{2n-1})(1-a_1a_2a_3a_4q^{2n-2})},\\[7pt]
C_{n+1}&=(1-q^{n+1})(1-a_1a_2a_3a_4q^{n-1}) \\[7pt]
&\quad\times \frac{(1-a_1a_2q^n)(1-a_1a_3q^n)(1-a_1a_4q^n)(1-a_2a_3q^n)(1-a_2a_4q^n)(1-a_3a_4q^n)}{4(1-a_1a_2a_3a_4q^{2n-1})(1-a_1a_2a_3a_4q^{2n})^2 (1-a_1a_2a_3a_4q^{2n+1})}\;.
\end{align*}
For our purpose we adopt the following definition.

\begin{definition}
A monic OPS, $(P_n)_{n\geq 0}$, with respect to the functional ${\bf u}\in \mathcal{P}^*$ is said classical if the sequence $(\mathcal{D}_qP_{n+1})_{n\geq 0}$ is orthogonal. This is also equivalent (see \cite[Theorem 5]{F}) to say 
\begin{align}\label{pearson-equation}
{\bf D}_q (\phi {\bf u})={\bf S}_q(\psi {\bf u})\;,
\end{align} 
where $\phi$ and $\psi$ are polynomials of degree at most two and one respectively.
\end{definition}

In the following lemma, we show that the polynomials $\phi$ and $\psi$ appearing in \eqref{pearson-equation} for classical OPS are only determined by the coefficients $B_0$, $B_1$, $C_1$ and $C_2$ provided by the TTRR satisfied by the OPS $(P_n)_{n\geq 0}$. 
We emphasize that this fact was also observed in \cite[Lemma 2.2]{KCDMJP2021c}.
 
\begin{lemma}\label{L1}
Let $(P_n)_{n\geq 0}$ be a classical monic OPS with respect to the functional ${\bf u}$. Let \eqref{TTRR} be the TTRR satisfied by $(P_n)_{n\geq 0}$.
Then the polynomials $\phi$ and $\psi$  appearing in distributional equation \eqref{pearson-equation} are given by $$\phi(x)=(\mathfrak{a}x-\mathfrak{b})(x-B_0)-(\mathfrak{a}+\alpha)C_1\;,~\psi(x)=x-B_0\;,$$
with 
\begin{align*}
\mathfrak{a}&=\frac{\alpha(3-4\alpha^2)}{4\alpha^2-1}+\frac{(B_0+B_1)^2+4\alpha^2(C_1-B_0B_1+\alpha^2-1)}{2\alpha(4\alpha^2-1)C_2}\;, \\ \mathfrak{b}&=(\mathfrak{a}+\alpha)B_1-\frac{B_0+B_1}{2\alpha}\;.
\end{align*}
\end{lemma}

\begin{proof}
Suppose that \eqref{pearson-equation} holds where we set $\phi(x)=ax^2+bx+c$ and $\psi(x)=x+e$, with $a,b,c,e\in \mathbb{C}$. We have 
\begin{align}
0= \left\langle {\bf D}_q (\phi {\bf u})-{\bf S}_q (\psi {\bf u}), ~x^n\right\rangle=-\left\langle {\bf u},~ \phi \mathcal{D}_q x^n +\psi\mathcal{S}_q x^n\right\rangle\;, n=0,1,\ldots\; \label{moments-functionals}
\end{align}
We recall that the coefficients $B_n$ and $C_n$ of the TTRR \eqref{TTRR} satisfied by $(P_n)_{n\geq 0}$ are given in the following form
\begin{align*}
B_n= \frac{\left\langle {\bf u},~ xP_n ^2\right\rangle}{\left\langle {\bf u},~ P_n ^2\right\rangle}\;, \quad C_{n+1}= \frac{\left\langle {\bf u},~ P_{n+1} ^2\right\rangle}{\left\langle {\bf u},~ P_n ^2\right\rangle}\;,\quad n=0,1,\ldots\;.
\end{align*}
The following identities can be computed easily
\begin{align*}
&\mathcal{D}_q x^2=2\alpha x\;,~\mathcal{D}_q x^3= (4\alpha^2-1)x^2+1-\alpha^2\;,\\
&\mathcal{S}_qx^2=(2\alpha^2-1)x^2+1-\alpha^2\;,~
\mathcal{S}_qx^3=\alpha(4\alpha^2-3)x^3+3\alpha(1-\alpha^2)x.
\end{align*}
In addition, we also have $1=P_0 ^2,~ x=xP_0 ^2,~ x^2=P_1 ^2 +2B_0xP_0 ^2 -B_0 ^2P_0 ^2$,  with $x^3=xP_1 ^2 +2B_0P_1 ^2 +3B_0 ^2xP_0 ^2 -2B_0 ^3P_0 ^2 $ and
\begin{align*}
x^4=& P_2 ^2 +2(B_0+B_1)xP_1 ^2 +\big[(B_0+B_1)(3B_0-B_1)-2(B_0B_1-C_1)\big]P_1 ^2 \\
&+2\big[B_0(B_0+B_1)(2B_0-B_1)+(B_0B_1-C_1)(B_1-B_0)   \big]xP_0 ^2 \\
&+\Big((B_0+B_1)(B_1-3B_0)B_0 ^2 -(B_0B_1-C_1)(B_0B_1-C_1-2B_0)  \Big)P_0 ^2\;.
\end{align*}
Therefore taking successively $n=0,1,2,3$ in \eqref{moments-functionals}, using what is preceding, we obtain
\begin{align*}
&e=-B_0, ~c=-(a+\alpha)C_1-(b+aB_0)B_0, ~b=-(a+\alpha)B_1 -aB_0+\frac{B_0+B_1}{2\alpha}\\
&a=\frac{\alpha(3-4\alpha^2)}{4\alpha^2-1}+\frac{(B_0+B_1)^2+4\alpha^2(C_1-B_0B_1+\alpha^2-1)}{2\alpha(4\alpha^2-1)C_2}\;,
\end{align*}
and the proof is completed.
\end{proof}

In the following theorem, we show how to recognize non-classical OPS. This was also partially observed in \cite[Lemma 2.3, equations (2.17) and (2.18)]{KCDMJP2021c}.

\begin{theorem}\label{T1}
Let $(P_n)_{n\geq 0}$ be a monic OPS satisfying the TTRR \eqref{TTRR}. If $(P_n)_{n\geq 0}$ is classical then the sequences $(B_n)_{n\geq 0}$ and $(C_n)_{n\geq 1}$ are solutions of the following system of difference equations
\begin{align}
&r_{n+3}B_{n+2} -(r_{n+2} +r_{n+1})B_{n+1}  +r_n B_n =0\;, \label{eq1S} \\
&\label{eq2S} r_n \left( B_{n} -qB_{n-1} \right)\left(B_{n}-q^{-1}B_{n-1}  \right)=(r_{n+1}+r_{n+2})(C_{n+1}-1/4)\\
&\nonumber \quad\quad \quad \quad\quad \quad\quad \quad \quad-4\alpha^2 r_n(C_n -1/4) +(r_{n-1}+r_{n-2})(C_{n-1} -1/4)\;,\nonumber 
\end{align}
where $(r_n)_{n\geq 0}$ is a nonzero complex sequence given by
\begin{align}\label{def-rn}
r_n= \widehat{a}q^n +\widehat{b}q^{-n}\;,\quad \widehat{a},\widehat{b}\in \mathbb{C}, \quad n=0,1,\ldots\;,
\end{align}
with $|\widehat{a}|+|\widehat{b}|\neq 0$.

\end{theorem}

\begin{proof}
Assuming that $(P_n)_{n\geq 0}$ is classical, $(P_n)_{n\geq 0}$ and $(\mathcal{D}_qP_{n+1})_{n\geq 0}$ are both OPS. This implies that (see \cite[Theorem 1.2]{KCDM2023}) $(P_n)_{n\geq 0}$ are Askey-Wilson polynomials or special or limiting cases of them. 
 For any set of complex numbers $\widehat{a}$ and $\widehat{b}$ in \eqref{def-rn} such that $$\widehat{a}q^2=-a_1a_2a_3a_4\widehat{b},$$ one can easily show that the monic Askey polynomials including special or limiting cases $(Q_{n}(.;a_1,a_2,a_3,a_4|q^{\pm 1}))_{n\geq 0}$ satisfy the system of equations \eqref{eq1S}--\eqref{eq2S}.
\end{proof}
\begin{remark}
We emphasize that the converse of Theorem \ref{T1} is not true. For instance consider a monic OPS whose coefficients of the TTRR are given by
$$B_n =0,\quad C_{n+1}=(1-aq^{n+1})(1-bq^{n+1})/4\;,\quad n=0,1,\ldots\;,$$ where $a$ and $b$ are nonzero complex numbers both different from one. These coefficients satisfy of \eqref{eq1S}--\eqref{eq2S} but the corresponding OPS is not classical.
\end{remark}

In \cite[Theorem 4.1]{KMP2022} necessary and sufficient conditions for the regularity of nonzero functionals satisfying \eqref{pearson-equation} and in addition the coefficients of the TTRR satisfied by the corresponding OPS are given.  From results and ideas developed therein and Lemma \ref{L1} we obtain the following theorem.

\begin{theorem}\label{T2}
Let $B_0$, $B_1$, $C_1$ and $C_2$ be four complex numbers and assume that the following sequences are well defined and with $C_n \neq 0$ for each $n=1,2,\ldots$.
\begin{align*}
B_n  = \frac{\gamma_{n+1}e_n}{d_{2n}}-\frac{\gamma_n e_{n-1}}{d_{2n-2}},\quad
C_{n+1}  =-\frac{\gamma_{n+1}d_{n-1}}{d_{2n-1}d_{2n+1}}\phi^{[n]}\left( \frac{e_{n}}{d_{2n}}\right),
\end{align*}
where $d_n=a\gamma_n+\alpha_n$, $e_n=(b+aB_0)\gamma_n+B_0\alpha_n$, and 
\begin{align*}
\phi^{[n]}(z)&=\big((\alpha^2-1)\gamma_{2n}+a\alpha_{2n}\big)
\big(z^2-1/2\big)-\big((b+aB_0)\alpha_n+(\alpha^2-1)B_0\gamma_n\big)z\\
&+bB_0 -(a+\alpha)C_1 +a/2,
\end{align*}
with $b=(a+\alpha)B_1 -(B_0+B_1)/(2\alpha)$ and 
\begin{align*}
a=\frac{\alpha(3-4\alpha^2)}{4\alpha^2-1}+\frac{(B_0+B_1)^2+4\alpha^2(C_1-B_0B_1+\alpha^2-1)}{2\alpha(4\alpha^2-1)C_2}\;.
\end{align*}
If $(P_n)_{n\geq 0}$ is a monic OPS whose coefficients of the TTRR are $(B_n)_{n\geq 0}$ and $(C_n)_{n\geq 1}$, then it is classical.
\end{theorem}

\section{Conclusion}
As we can see, given an OPS only defined by its TTRR, from Theorem \ref{T1} and Theorem \ref{T2} it is possible by hands to determine whether they are classical or not. Once they are classical, all necessary informations about the OPS (corresponding weight function, Rodrigues-type formulas, integral representations, zeros and many other properties and characterizations) are in the literature (see \cite{KMP2022, I2005, KLS2010} and references therein). This is much simple than the idea and method presented in \cite{DDT}.

\section*{Acknowledgments}
The author was partially supported by CMUP, member of LASI, which is financed by national funds through FCT - Fundac\~ao para a Ci\^encia e a Tecnologia, I.P., under the projects with reference UIDB/00144/2020 and UIDP/00144/2020.

\end{document}